\documentclass[]{amsart}

\usepackage{amsmath,amssymb,amsthm}
\usepackage{graphicx}
\newtheorem{lemma}{Lemma}[section]
\newtheorem{cor}[lemma]{Corollary}

\newtheorem{proposition}[lemma]{Proposition}
\theoremstyle{definition}

\newtheoremstyle{examplestyle}            
     {.8cm}                    
     {1.0cm}                    
     {\upshape}                    
     {}                        
     {\bfseries\slshape}                
     {.}                    
     {  }                    
     {\thmname{#1}\thmnumber{ #2}\thmnote{ (#3)}}                        
\theoremstyle{examplestyle}
\newtheorem{example}[lemma]{Example}

\numberwithin{equation}{section}
\newcommand{\fst}[1][{[0,1]}]{\text{ f.s. } t\in#1}
\newcommand{\fat}[1][{[0,1]}]{\text{ f.a. } t\in#1}
\usepackage{graphicx}
\usepackage{bm}
\usepackage{color}%

\newcommand{\norm}[1]{\left\Vert#1\right\Vert}
\newcommand{\abs}[1]{\left\vert#1\right\vert}

\newcommand{\R}{\mathbb R}
\newcommand{\N}{\mathbb N}
\newcommand{\Q}{\mathbb Q}

\newcommand{\bfone}{\bm{1}}

\newcommand{\bfZ}{\bm{Z}}

\newcommand{\bfzeta}{\bm{\zeta}}

\newcommand{\bfeta}{\bm{\eta}}

\newcommand{\barE}{\bar E^-[0,1]}
\newcommand{\barC}{\bar C^-[0,1]}

\allowdisplaybreaks[4]

\begin{document}

 \title{On the Hitting Probability of Max-Stable Processes}

\author{Martin Hofmann}
 \address{ University of W\"{u}rzburg \\
              Institute of Mathematics\\
               Emil-Fischer-Str. 30\\
                97074 W\"{u}rzburg, Germany\newline hofmann.martin@mathematik.uni-wuerzburg.de}
\subjclass{Primary 60G70}%
\keywords{Max-stable process, hitting probability, functional $D$-norm, total dependence process, probability of hitting more than once}
\begin{abstract}
The probability that a max-stable process $\bfeta$ in $C[0,1]$ with identical marginal distribution function $F$ hits $x \in \R$ with $0<F(x)<1$ is the hitting probability of $x$. We show that the hitting probability is always positive, unless the components of $\eta$ are completely dependent. Moreover, we consider the event that the paths of standard MSP hit some $x\in\R$ twice and give a sufficient condition for a positive probability of this event.
\end{abstract}

\maketitle

 \section{Introduction and Preliminaries}
 A \emph{max-stable process} (MSP) $\bfzeta=\left(\zeta_t\right)_{t\in[0,1]}$  which realizes in the space $C[0,1]:=\{f:[0,1]\to\R:\ f \textrm{ continuous}\}$, equipped with the sup-norm $\norm f_\infty=\sup_{t\in[0,1]}\abs{f(t)}$, is a stochastic process with the characteristic property that its distribution is max-stable, i.e., $\bfzeta$ has the same distribution as $\max_{1\leq i\leq n}(\bfzeta_i-b_n)/a_n$ for independent copies $\bfzeta_1,\bfzeta_2,\dots$ of $\bfzeta$ and some $a_n,b_n\in C[0,1],\,a_n>0$, $n\in\N$ (cf. de Haan and Ferreira  \cite{dehaf06}), where the maximum is taken pointwise.

As in the finite dimensional case, it is possible to consider some standard case of univariate marginal distributions and reaching all other cases via transformation of the univeriate margins, cf. \cite{ginhv90}, \cite{dehaf06}. In this paper we say that an MSP $\bfeta$ is a
\emph{standard} MSP, if it is an MSP with standard negative
exponential (one-dimensional) margins, $P(\eta_t\le x)=\exp(x)$,
$x\le 0$, $t\in[0,1]$.

In the following, we show that the paths of such standard MSP hit every $x_0<0$ with positive probability, unless the margins are completely dependent. In Section \ref{sec:more_than_once} we go beyond and consider the event that the sample paths of standard MSP hit some $x_0<0$ more than once.

The abbreviations "f.s." and "f.a." mean "for some" and "for all", respectively, for example $P(\eta(t)\leq f(t),\fat)=P(\eta(t)\leq f(t),\text{ for all }t\in[0,1])$.
 
 Denote by $E[0,1]$ the set of all functions on $[0,1]$ which are
bounded and which have only a finite number of discontinuities and by $\bar E^-[0,1]$ those functions in $E[0,1]$ which do not attain positive values.

 Due to Aulbach et. al \cite{aulfaho11}, there is for every standard MSP  $\bfeta$  some continuous \emph{generator process} $\bfZ=(Z_t)_{t\in[0,1]}$ in $C[0,1]$ with the properties 
 \begin{equation}\label{eqn:properties_of_generator}
\bfZ\ge 0 \text{ a.s.},\quad  E(Z_t)=1,\  t\in[0,1], \quad\mbox{ and }\quad m:=E(\sup_{t\in[0,1]}Z_t) < \infty.
\end{equation} The connection between $\bfeta$ and $\bfZ$ is 
\begin{eqnarray}
P(\bfeta_t\le f(t),\fat) &=& \exp\left(-E\left(\sup_{t\in [0,1]}\left(\abs{f(t)}Z_t\right)\right)\right)\nonumber\\
&=:& \exp\left(-\norm f_D\right),\label{eq:distribution_function_of_standard_MSP}
\end{eqnarray}
which holds for every $ f\in \bar E^-[0,1]$.
Conversely, if there is some continuous process $\bfZ$ with properties \eqref{eqn:properties_of_generator} then there exists a standard MSP $\bfeta$ with this generator, cf. \cite{aulfaho11}.

It is easy to see that $\norm \cdot_D$ defines a norm on  $\bar E^-[0,1]$, and it is called the $D$-norm of $\bfeta$. While a generator $\bfZ$ is not uniquely determined by equation \eqref{eq:distribution_function_of_standard_MSP}, the generator constant $m=\norm{1}_D$ obviously is.

\section{Hitting Probability of standard MSP}
The considerations in Aulbach et. al \cite{aulfaho11} entail in particular, that for every subinterval $I\subset[0,1]$ of positive length 
\begin{equation}\label{eq:standardMSPnotzero}
P(\eta_t=0,\fst[I])=1-P\left(\sup_{t\in I}\eta_t<0\right)=0.
\end{equation}
Furthermore, there is for every $f\in\barE$ 
\begin{equation*}
P\left(\bfeta\leq f\right)\ =\ P\left(\bfeta < f\right),
\end{equation*}
which follows immediately from the fact, that every $D$-norm is equivalent to the sup-norm on $\barE$, cf. \cite{aulfaho11}.
But this implies for all $f\in\barE$
\begin{eqnarray}
\lefteqn{P\left(\left\{\eta_t=f(t),\fst\right\}\cap\left\{\eta_t\leq f(t),\fat\right\}\right)}\notag\\
&=& P\left( \bfeta (t)\leq f(t),\fat\right)-P\left( \eta_t < f(t),\fat\right)\notag\\
&=&0.\label{eq:margin_prob_is_zero}
\end{eqnarray}
Now one may raise the question wether $P\left(\eta_t=f(t),\fst\right)=0$ is true for $f\in\barE$, in accordance to the finite dimensional case: as a rv $X=(X_1,\ldots,X_d)$ in $\R^d$ with negative exponentially distributed margins has a continuous distribution, there is $P\left(X_i=x, \text{ for some }i\in\{1,\ldots,d\}\right)=0$ for every $x\in(-\infty,0]$.

\begin{example}\label{ex:completedependence}
Consider the complete dependence case, i.e. the standard MSP $\bfeta$ with generator constant $m=1$ (note that this is the case iff the corresponding $D$-norm is equal to the sup-norm, $\norm{\cdot}_D=\norm{\cdot}_\infty$, due to the functional version of Takahashi's Theorem \cite{taka88}, \cite{aulfaho11}).
 We get immediately $P\left(\eta_t=x,\fst\right)=P(\eta_0=x)=0$, as $P(\eta_0\le x)=\exp(x),\ x<0.$

On the other hand, for arbitrary \emph{non-constant} continuous functions $f\in\barC$, 
\begin{equation*}
P\left(\eta_t=f(t),\fst\right)=P\left(\eta_0\in \mathfrak{im}(f)\right)>0,
\end{equation*}
as the image $\mathfrak{im}(f)$ of $f$ is an interval of positive length.
\end{example}

The next Proposition is the main result of this paper and gives a complete answer to the foregoing question. 

For some subset $I\subset[0,1]$ define by $\bfone_I:[0,1]\to\{0,1\}$ the indicator function of $I$, i.e. $\bfone_I(t)=1$, if $t\in I$, and $\bfone_I(t)=0$, if $t\not\in I$.

\begin{proposition}\label{theo:hittingProbabilityNonZero}
Let $\bfeta$ be a standard MSP with generator process $\bfZ$. Suppose that there exists $x_0<0$ and a subinterval $I\subset[0,1]$ with positive length such that 
\begin{equation*}
P(\eta_t=x_0,\fst[I]) = 0.
\end{equation*}
Then $Z_t=Z_s$ almost surely for $t,s\in I$. 

Conversely, if for a subinterval $I\subset[0,1]$ with positive length there is 
$\norm{\bfone_I}_{D}>1$, then 
\begin{equation*}
P(\eta_t=x_0, \fst[I]) > 0\ \text{ for all }x_0 < 0.
\end{equation*}
\end{proposition}

\begin{proof}
\allowdisplaybreaks
Assume $P\left(\eta_t=x_0,\fst[I]\right)=0$ for some $x_0<0$ and an interval $I\subset[0,1]$ with positive length. Define for $k\in\N$ and arbitrary $t_0\in I$ the functions $g,g_k\in\barE$ by
\begin{equation*}
g(t):= x_0\bfone_I(t); \qquad g_{t_0,k}(t):= (x_0-1/k)\bfone_{t_0}(t).
\end{equation*}

Then, with equation \eqref{eq:standardMSPnotzero},
\begin{eqnarray*}
P\left(\eta_t< x_0,\fat[I]\right)&=&P\left(\eta_t< g(t),\fat\right)\\
&=&\exp\left(x_0\norm{\bfone_I}_D\right)=\exp\left(x_0\left(\sup_{t\in I}Z_t\right)\right).
\end{eqnarray*}
By assumption, we get on the other hand
\begin{eqnarray*}
\exp(x_0-1/k)&=&P\left(\eta_t\leq g_{t_0,k}(t),\fat\right)\\
&=&P\left(\eta_t\leq g_{t_0,k},\fat, \eta_t<x_0,\fat[I]\right),
\end{eqnarray*}
and, thus,
\begin{eqnarray*}
\exp(x_0)&=& \lim_{k\to\infty}P\left(\eta_t\leq g_{t_0,k}(t),\fat, \eta_t<x_0,\fat[I]\right)\\
&=& P\left(\bigcup_{k\in\N}\{\eta_t\leq g_{t_0,k}(t),\fat, \eta_t<x_0,\fat[I]\}\right)\\
&=& P\left(\eta_t<x_0,\fat[I]\}\right)\\
&=& \exp\left(x_0\norm{\bfone_I}_D\right),
\end{eqnarray*}
i.e. $\norm{\bfone_I}_D=E\left(\sup_{t\in I}Z_t\right)=1$.

As $\bfZ$ is a generator process fulfilling the conditions \eqref{eqn:properties_of_generator}, we get for every $s\in I$
\begin{equation*}
E\left(\sup_{t\in I}Z_t-Z_s\right)=0 \iff \sup_{t\in I}Z_t=Z_s\ \text{a.s.},
\end{equation*}
and, thus, $Z_t=Z_s$ for all $s,t\in I$ with probability one.
\end{proof}

We give an example of a standard MSP $\bfeta$, which has a generator constant $m>1$ but there is a interval $I$ on which its generator $\bfZ$ fulfills $Z_t=Z_s$ for all $s,t\in I$ a.s..

\begin{example}\label{Z_partially_constant}
Let $Z_0,Z_1$ some independent and identical distributed rv with 
$$
P(Z_i=\frac1n)=\frac{n}{n+1}=1-P(Z_i=n),\text{ i.e. }E(Z_i)=1,\ i=0,1,
$$ for some $n\in\N$. With some $0<a<b<1$ define
\begin{equation*}
Z_t:=\begin{cases}\frac{a-t}{a}Z_0+\frac ta & \text{for }t\in[0,a);\\
                  1                         & \text{for }t\in[a,b];\\
                  \frac{1-t}{1-b}+ \frac{t-b}{1-b}Z_1& \text{for }t\in(b,1],\end{cases}
\end{equation*}
and $\bfZ$ is obviously a generator process fulfilling \eqref{eqn:properties_of_generator}. But $\bfeta$ is for $n>1$ not the complete dependence MSP as $m=E(\sup_{t\in[0,1]}Z_t)=(3n^2+n)/(n+1)^2>1$ for  $n>1$.
\end{example}

The following Corollary follows immediately from Proposition \ref{theo:hittingProbabilityNonZero}.

\begin{cor}
A standard MSP $\bfeta$ has complete dependent margins, i.e. its $D$-norm is equal to the $\sup$-norm, if and only if
\begin{equation}
P\left(\eta_t=x_0,\fst\right)=0
\label{eq:characterizationofComplDep}
\end{equation}
 for some $x_0<0$. In this case \eqref{eq:characterizationofComplDep} holds for \emph{every} $x_0<0$.
\end{cor}

Now we consider for some standard MSP $\bfeta$ the function $h_{\bfeta}:(-\infty,0]\to[0,1]$ defined by $x\mapsto P\left(\eta_t=x,\fst\right).$

Then $h_{\bfeta}(x)$ is the "hitting probabilitiy" of $\bfeta$ and $x$. Proposition \ref{hittingprob} below states some properties of $h_{\bfeta}$. Its proof uses the next lemma, which is established in Aulbach et al. \cite{aulfaho11}.

\begin{lemma} If $\bfeta$ is a standard MSP with generator $\bfZ$, we have for $f\in\barE$
\begin{equation*}
P(\eta_t>f(t),\fat) \ge 1-\exp\left(-E\left(\inf_{0\le t\le
1}(\abs{f(t)}Z_t)\right)\right).
\label{eq:expansion_of_survivor_function_of_eta}
\end{equation*}
\end{lemma}

\begin{proposition}\label{hittingprob}
Let $\bfeta$ a standard MSP with generator $\bfZ$, generator constant $m=E(\sup_{t\in[0,1]}Z_t)>1$ and with the additional property that $\widetilde m:=E(\inf_{t\in[0,1]}Z_t)>0$. 
Then the hitting probability $h_{\bfeta}$ has the properties
\begin{equation*}
h_{\bfeta}(0)=0, \qquad h_{\bfeta}(x)>0 \text{ for } x<0 \quad\text{and}\quad \lim_{x\to-\infty}h_{\bfeta}(x)=0.
\end{equation*}
Moreover,
\begin{equation*}
 0<\int_{-\infty}^0 h_{\bfeta}(x)\,dx\leq \frac{m-\widetilde m}{m\widetilde m}.
\end{equation*}
\end{proposition}

\begin{proof}
The assertion follows by Proposition \ref{theo:hittingProbabilityNonZero} and the inequality 
\begin{eqnarray*}
\lefteqn{P\left(\eta_t=x,\fst\right)}\\
&=& 1-P\left(\eta_t\neq x,\fat\right)\\
&=& 1- \left[ P\left(\eta_t>x,\fat\right)+P\left(\eta_t<x,\fat\right)\right]\\
&=& P\left(\eta_t\leq x,\fst\right)-\exp(xm)\\
&\leq&\exp(x\widetilde m)  -\exp(xm),
\end{eqnarray*}
which holds for all $x\in(-\infty,0]$ by Lemma  \ref{eq:expansion_of_survivor_function_of_eta}.
\end{proof}

In the setup of the preceding proposition, the term $\frac{m-\widetilde m}{m\widetilde m}$ can be interpreted as a measure of the dependence structure of $\bfeta$. In case of complete dependence we have $m=\widetilde m=1$, and, thus, $\frac{m-\widetilde m}{m\widetilde m}=0.$ In case of $m>1$ we immediately get $\widetilde m<1$ and, therefore, $\frac{m-\widetilde m}{m\widetilde m}>0$.

\section{Probability of hitting more than once}\label{sec:more_than_once}

Now the question arises how often the paths of a standard MSP hit some $x_0<0$. We give a sufficient condition on the generator $\bfZ$ of a standard MSP $\bfeta$ such that the probability of the event that the paths of $\bfeta$ hit some $x_0<0$ (at least) two times in some interval $[t',t'']\subset[0,1]$ is positive for every $x_0<0$. We need the following Lemma which is of interest of its own.

\begin{lemma}\label{lem:specialhitting}
Take $0\leq t'<t''\leq 1$ and consider a standard MSP $\bfeta$ with generator $\bfZ$. Then, for every $t_0\in(t',t'')$ and  every $x_0<0$
\begin{equation*}
 P\left(\eta_{t'}\leq x_0,\eta_{t_0}>x_0,\eta_{t''}\leq x_0\right) =0 
\end{equation*}
if, and only if,  
\begin{equation*}
E\left(\sup_{t\in[t',t'']}Z_t\right) = E\left(\max(Z_{t'},Z_{t''})\right).
\end{equation*}
\end{lemma}

\begin{proof}
Let $t_0\in(t',t'')$ and $x_0<0$ be given. Then
\begin{eqnarray*}
\lefteqn{P\left(\eta_{t'}\leq x_0,\eta_{t_0}>x_0,\eta_{t''}\leq x_0\right)}\\
&=& P\left(\eta_{t'}\leq x_0 ,\eta_{t''}\leq x_0\right) -P\left(\eta_{t'}\leq x_0,\eta_{t_0}\leq x_0,\eta_{t''}\leq x_0\right)\\
&=& \exp\left(x_0E\left(\max(Z_{t'},Z_{t''})\right)\right)-\exp\left(x_0E\left(\max(Z_{t'},Z_{t_0},Z_{t''})\right)\right),
\end{eqnarray*}
and this is equal to zero if, and only if, 
\begin{eqnarray*}
&&E\left(\max(Z_{t'},Z_{t''})\right)=E\left(\max(Z_{t'},Z_{t_0},Z_{t''})\right)\\
 \iff&& P\left(\max(Z_{t'},Z_{t''})=\max(Z_{t'},Z_{t_0},Z_{t''})\right)=1.
\end{eqnarray*}
Since $t_0\in(t',t'')$ was arbitrary, we get for finitely many $t_1,\ldots,t_n\in (t',t''),\ n\in\N,$
\begin{equation*}
\max(Z_{t'},Z_{t_1},\ldots,Z_{t_n},Z_{t''})= \max(Z_{t'},Z_{t''}) \quad \text{a.s.},
\end{equation*} 
and, thus, the continuity of $\bfZ$ implies with $\{t_1,t_2,\ldots\}:=(t',t'')\cap\Q$
\begin{equation}\label{eq:generator_condition}
\sup_{t\in[t',t'']}Z_t=\sup_{t\in\{t',t'',t_1,t_2,\ldots\}}Z_t=\lim_{n\to\infty}\max_{t\in\{t',t_1,\ldots,t_n,t''\}}Z_t=\max(Z_{t'},Z_{t''})
\end{equation}
with probability one, which is equivalent to $E\left(\sup_{t\in[t',t'']}Z_t\right) = E\left(\max(Z_{t'},Z_{t''})\right)$.
\end{proof}

Now the following assertion on the hitting probability follows easily from the foregoing Lemma.
\begin{proposition} \label{prop:suff_cond_hitting2times}
Let $t',t''\in[0,1]$ be arbitrary with $0\leq t'<t''\leq 1$ and consider a standard MSP $\bfeta$ with generator $\bfZ$. If we have 
\begin{equation}\label{equ:expectation_sup=max}
E\left(\sup_{t\in[t',t'']}Z_t\right)> E\left(\max\left(Z_{t'},Z_{t''}\right)\right),
\end{equation} then, for every $t_0\in(t',t'')$ and every $x_0<0$,
\begin{equation*}
P\left(\eta_t=x_0 \fst[[t',t_0]],\eta_t=x_0 \fst[[t_0,t'']]\right)>0.
\end{equation*}
\end{proposition}

\begin{proof}
 By Lemma \ref{lem:specialhitting}, condition \eqref{equ:expectation_sup=max} is equivalent to $P\left(\eta_{t'}\leq x_0,\eta_{t_0}>x_0,\eta_{t''}\leq x_0\right)>0$, and, thus, 
\begin{eqnarray*}
\lefteqn{P\left(\eta_t=x_0 \fst[[t',t_0]],\eta_t=x_0 \fst[[t_0,t'']]\right) }\\
&\geq &P\left(\eta_{t'}\leq x_0,\eta_{t_0}>x_0,\eta_{t''}\leq x_0\right) >0.
\end{eqnarray*}

\end{proof}

In the proof of Lemma \ref{lem:specialhitting}, the property $\sup_{t\in[t',t'']}Z_t=\max(Z_{t'},Z_{t''})$ almost surely of a generator process $\bfZ$ plays a crucial role, cf. equation \eqref{eq:generator_condition}. It is clear that a generator process which is pathwise linear on $[t',t'']$, i.e. $Z_t:=\frac{t''-t}{t''-t'}Z_{t'}+ \frac{t-t'}{t''-t'}Z_{t''}, t\in [t',t'']$ a.s.,  obviously fulfills \eqref{eq:generator_condition}. All paths of a generator $\bfZ$ fulfilling \eqref{eq:generator_condition} have to be either strictly monotone or convex on $[t',t'']$ and one may ask if there are other examples than pathwise linear processes: the answer is "yes", as the next Example shows.
\begin{example}\label{ex:nonlineargen}
Take real numbers $a,b,c,d,e>0$ with the following properties:
\begin{equation*}
1<a;\quad b<1;\quad 1<c<\frac{a-b}{a-1};\quad (1<)\frac{a-b}{a-b-c(a-1)}<d;\quad e<1;
\end{equation*}
and define
\begin{equation*}
p:= \frac{1-b}{a-b};\quad \widetilde p:= \frac{1-e}{d-e}.
\end{equation*}
Let $Y,\widetilde Y$ be some independent Bernoulli rvs with $P(Y=1)=p=1- P(Y=0)$ and $P(\widetilde Y=1)=\widetilde p=1- P(\widetilde Y=0)$.
Now define
\begin{equation*}
Z_0:= Ya+(1-Y)b;\quad Z_{1/2}:=1;\quad Z_1:=(1-Y)c+\left(1- \frac{a-1}{a-b}c\right)\left(\widetilde Yd+(1-\widetilde Y)e\right).
\end{equation*}
Elementary computations show that $E(Z_0)=E(Z_{1/2})=E(Z_1)=1$, so the linear interpolation process $\bfZ=(Z_t)_{t\in[0,1]}$ defined by
\begin{equation*}
Z_t:=\begin{cases} 2(\frac12-t)Z_0 +2tZ_{1/2} & \text{ for } t\in[0,1/2]\\
                   2(1-t)Z_{1/2}+ 2(t-\frac12)Z_1& \text{ for } t\in[1/2,1]. \end{cases}
\end{equation*} 
is a proper generator process. We have 
\begin{equation*}
P\left(\sup_{t\in[t',t'']}Z_t = \max(Z_{t'},Z_{t''})\right)=1
\end{equation*}
for \emph{arbitrary} $0\le t'<t''\le 1$, as three of the possible four paths are (strictly) monotone, and there is (with probability $p\cdot\widetilde p$) one path which is (strictly) convex.
Note that the numbers $a,b,c,d,e$ can be substituted by appropriate rvs (independent of each other and of $Y,\widetilde Y$), which have those values as their expectation, respectively, and that $Z_{1/2}$ can also be chosen to be random.
\end{example}

Nevertheless, equation \eqref{eq:generator_condition} has some further implications.

\begin{cor}\label{cor:generator_condititon}
Let $\bfeta$ be a standard MSP with generator process $\bfZ$ and fix $0\le t'<t''\le 1$. Then the following conditions are equivalent:
\begin{enumerate}
    \item \label{item:hitting_prob}$P\left(\eta_{t'}\leq x_0,\eta_{t_0}>x_0,\eta_{t''}\leq x_0\right)=0$, for all $x_0<0$ and every $t_0\in(t',t'')$;
	\item \label{item:generator_condition} $P(\sup_{t\in[t',t'']}Z_t=\max(Z_{t'},Z_{t''}))=1$;
	\item \label{item:generator_constant_condition} $E(\sup_{t\in[t',t'']}Z_t)=E(\max(Z_{t'},Z_{t''}))$;
	\item \label{item:df}$P(\eta_t\leq x_0\fat[[t',t'']])=P(\eta_{t'}\leq x_0,\eta_{t''}\leq x_0)$, for all $x_0<0$;
	\item \label{item:survivor_prob} $P(\eta_t\leq x_0\fat[[t',t'']])-P(\eta_{t'} > x_0,\eta_{t''} > x_0)= 2\exp(x_0)-1$, for all $x_0<0$.
\end{enumerate}
\end{cor} 
\begin{proof}
The proof of Proposition \ref{prop:suff_cond_hitting2times} already contains \ref{item:hitting_prob}$\iff$\ref{item:generator_condition}$\iff$\ref{item:generator_constant_condition}. Moreover, \ref{item:generator_constant_condition} is true, if, and only if,
\begin{eqnarray*}
 P(\eta_t\leq x_0\bfone_{[t',t'']}(t) \fat)
&=& \exp\left(x_0E\left(\sup_{t\in[t',t'']}Z_t\right)\right)\\
&=& \exp\left(x_0E\left(\max(Z_{t'},Z_{t''}\right)\right)\\
&=& P(\eta_{t'}\leq x_0,\eta_{t''}\leq x_0),
\end{eqnarray*} 
 for all $x_0<0$, i.e. \ref{item:generator_constant_condition}$\iff$\ref{item:df}. Finally, we have for arbitrary $x_0<0$
 \begin{eqnarray*}
\lefteqn{P(\eta_{t'}\leq x_0,\eta_{t''}\leq x_0)=}&&\\
&=& P(\eta_t\leq x_0\bfone_{[t',t'']}(t) \fat) \\
&& \qquad + P(\{\eta_{t'}\leq x_0,\eta_{t''}\leq x_0\}\cap\{\eta_t> x_0 \fst[t',t'']\}),
\end{eqnarray*}
and, thus,
\begin{eqnarray*}
\text{\ref{item:df}}&\iff&P(\{\eta_{t'}\leq x_0,\eta_{t''}\leq x_0\}\cap\{\eta_t> x_0 \fst[t',t'']\})=0\\
&\iff&P(\{\eta_{t'}> x_0\}\cup\{\eta_{t''}> x_0\}\cup\{\eta_t\le x_0 \fat[t',t'']\})=1,
\end{eqnarray*}
which is \ref{item:survivor_prob} by using the inclusion-exclusion formula.
\end{proof}

We finish with an example which shows in particular, that there are standard MSP, which hit every $x_0<0$ twice with positive probability, but the probability of hitting any $x_0<0$ three or more times is equal to zero.  
\begin{example} Let $\eta_0, \eta_1$ independent negative exponantial distributed rvs and define the continuous process $\bfeta$ by
 \begin{equation*}
\eta_t:= \max(\frac{1}{1-t}\eta_0,\frac1t\eta_1),\quad t\in[0,1].
\end{equation*}

 Elementary computations show, that all fidis of $\bfeta$ are max-stable and that the one-dimensional marginal distributions are standard negative exponantial, so $\bfeta$ is a standard MSP.
 
 We have $P(\eta_t<x,\fat)=P(\max(\eta_0,\eta_1)<x)=\exp(2x)$ for $x<0$, so the generator constant of $\bfeta$ is given by $m=2$.
 
 Moreover, elementary computations yield for abitrary $x<0$:
 \begin{eqnarray*}
h(x)&=&P(\eta_t=x,\fst) \\
&=&(1-\exp(x)-x)\exp(x),
\end{eqnarray*}
and this implies $\int_{-\infty}^0h(x)\,dx=3/2$.

Furthermore, in this example elementary computations yield for every $t_0\in(0,1)$ and arbitrary $x_0<0$
\begin{eqnarray*}
\lefteqn{P\big(\{\eta_t=x_0,\fst[[0,t_0)]\}\cap\{\eta_t=x_0,\fst[[t_0,1]]\}\big)}\\
&\qquad \qquad=& (\exp(x_0(1-t_0))-\exp(x_0))(\exp(x_0t_0)-\exp(x_0))>0,
\end{eqnarray*}
so all paths of $\bfeta$ hit every $x_0<0$ two times with positive probability.

On the other hand, it can be shown by elementary arguments that we have for disjoint intervals $I_1,I_2,I_3\subset[0,1]$
\begin{equation*}
P\big(\bigcap_{k=1,2,3}\{\eta_t=x_0,\fst[I_k]\}\big)=0,\ x_0<0,
\end{equation*}
so every path of $\bfeta$ does not hit any $x_0<0$ three times (or more often).
\end{example}

 \section*{Acknowledgement}
 I am very grateful to Michael Falk for constructive and fruitful discussions on the issue of this paper.

\end{document}